\documentclass[12pt,leqno]{amsart}
\usepackage{etoolbox}
\usepackage{graphicx}
\usepackage[english]{babel}
\usepackage[utf8]{inputenc}
\usepackage[noadjust]{cite}
\usepackage{amsmath, amsthm}
\usepackage{amsfonts, amssymb}
\usepackage{bm}
\usepackage{tikz}
\usetikzlibrary{arrows,automata,calc}
\usepackage{pgfplots}
\pgfplotsset{compat=1.17}
\usepackage[T1]{fontenc}
\usepackage[tracking]{microtype}
\usepackage[ttscale=0.81]{libertine}
\usepackage[libertine]{newtxmath}
\usepackage{geometry}

\usepackage{hyperref}
\hypersetup{
  colorlinks=true,
  linkcolor=green!30!black,
  linktoc=all,
  citecolor=green!50!black,
  bookmarksnumbered=true
}

\numberwithin{equation}{section}
\usepackage{thmtools}
\usepackage{thm-restate}

\newtheorem{theorem}{Theorem}[section]

\newtheorem{lemma}[theorem]{Lemma}
\newtheorem{corollary}[theorem]{Corollary}

\newtheorem{question}[theorem]{Question}
\theoremstyle{definition}

\newtheorem{remark}[theorem]{Remark}

\newcommand{\Rel}[2]{\operatorname{Rel}(#1;\, #2)}
\newcommand{\sRel}[3]{\operatorname{splitRel}_{#2}(#1;\, #3)}

\begin{document}

\title{Density of reliability roots of simple graphs in the unit disk}
\author{Pjotr Buys}
\address{Korteweg-de Vries Institute for Mathematics, University of Amsterdam}
\email{pjotr.buys@gmail.com}
\date{\today}

\begin{abstract}
Brown and Colbourn (1992) showed that the complex roots of the reliability polynomial of connected multigraphs are dense in the unit disk and that the closure of the real roots is $[-1,0] \cup \{1\}$. We prove the simple graph analogues of both results, confirming a recent conjecture of Brown and McMullin. The proof uses the family of graphs $C_m[K_n]$ obtained by substituting each edge of a cycle $C_m$ with a complete graph $K_n$, and relies on the asymptotic behavior of the reliability and split reliability polynomials of $K_n$.
\end{abstract}

\maketitle

\section{Introduction}

Let $G = (V,E)$ be a connected (multi-)graph. If each edge fails independently with probability $q$, the probability that the operational edges still form a connected spanning subgraph is
\[
	\Rel{G}{q} = \sum_{\substack{S \subseteq E \\ (V,S) \text{ connected}}} (1-q)^{|S|} q^{|E| - |S|}.
\]
This polynomial in $q$ is called the (all-terminal) \emph{reliability polynomial} of $G$. The reliability polynomial was introduced in~\cite{moore1956reliable} for two-terminal networks and generalized to arbitrary coherent systems in~\cite{birnbaum1961multicomponent}, of which all-terminal reliability is a special case; see~\cite{colbourn1987combinatorics} for a comprehensive treatment.

Since $\Rel{G}{q}$ is a polynomial, one can consider its complex roots. Define
\[
	\mathcal{Z}_{\mathrm{multi}} = \{q \in \mathbb{C} : \Rel{G}{q} = 0 \text{ for some connected multigraph } G\}
\]
\[
	\mathcal{Z}_{\mathrm{simple}} = \{q \in \mathbb{C} : \Rel{G}{q} = 0 \text{ for some connected simple graph } G\}.
\]
Motivated by various unimodality conjectures for the coefficient sequences of the reliability polynomial, Brown and Colbourn~\cite{brown1992roots} initiated the study of these roots and conjectured that all reliability roots lie in the closed unit disk, i.e.\ $\mathcal{Z}_{\mathrm{multi}} \subseteq \overline{\mathbb{D}}$. Using cycles with bundled edges (see \autoref{fig:constructions}, left), they also showed that reliability roots of multigraphs are dense in the unit disk, i.e.\ $\overline{\mathbb{D}} \subseteq \overline{\mathcal{Z}_{\mathrm{multi}}}$. Using the same family, they were also able to determine that the closure of the real reliability roots is given by $\overline{\mathcal{Z}_{\mathrm{multi}} \cap \mathbb{R}} = [-1,0] \cup \{1\}$.

Royle and Sokal~\cite{royle2004brown} disproved the Brown-Colbourn conjecture by exhibiting a simple planar graph with zeros strictly outside the unit disk. Zeros of even larger modulus were subsequently found in~\cite{brown2017roots}. Moreover, Bencs, Piombi, and Regts~\cite{bencs2025zeros} showed that $\overline{\mathcal{Z}_{\mathrm{multi}}}$ contains an open neighborhood of every $q \in \partial \mathbb{D}$ with $q^k \neq 1$ for $k = 1, \ldots, 4$. This means that $\mathcal{Z}_{\mathrm{multi}} \not\subseteq \mathbb{D}$ in almost every direction. Whether $\mathcal{Z}_{\mathrm{multi}}$ is bounded remains open.

As mentioned, simple graphs can also have reliability roots outside the unit disk. Regarding density inside the disk, however, less is known than for multigraphs. The sets $\mathcal{Z}_{\mathrm{multi}}$ and $\mathcal{Z}_{\mathrm{simple}}$ do not coincide, e.g.\ $-1 \in \mathcal{Z}_{\mathrm{multi}} \setminus \mathcal{Z}_{\mathrm{simple}}$~\cite{brown2020rational}. Brown and McMullin~\cite{brown2026realreliabilityrootsgraphs} proved that the real reliability roots of simple graphs are dense on an explicit subinterval $[\beta, 0] \subset [-1,0]$, where $\beta \approx -0.5707$, and conjectured that $\overline{\mathcal{Z}_{\mathrm{simple}} \cap \mathbb{R}} = [-1,0] \cup \{1\}$. The purpose of this paper is to confirm this conjecture and more generally prove the simple graph analogue of~\cite[Prop.~5.1]{brown1992roots}.\footnote{While finalizing this manuscript, we learned that Omar~\cite{omar2026real} has independently proved item~2 of \autoref{thm:main}.}
\begin{restatable}{theorem}{maintheorem}\label{thm:main}
	The reliability roots of simple graphs satisfy the following.
	\begin{enumerate}
		\item Reliability roots of simple graphs are dense in the unit disk, i.e. $\overline{\mathbb{D}} \subseteq \overline{\mathcal{Z}_{\mathrm{simple}}}$.
		\item Real reliability roots are dense in $[-1,0]$ and thus 
			$\overline{\mathcal{Z}_{\mathrm{simple}} \cap \mathbb{R}} = [-1,0] \cup \{1\}$.
	\end{enumerate}
\end{restatable}
\subsection{Split reliability}\label{sec:splitrel}
Many results on the location of reliability roots involve replacing edges of a host graph by another two-terminal graph, sometimes called a gadget. A two-terminal graph is a connected graph $H$ with two distinct distinguished vertices $u$ and $v$. Given a graph $G$ and a two-terminal graph $(H,u,v)$, we write $G[H(u,v)]$ for the graph obtained by replacing each edge of $G$ with a copy of $H$, identifying the endpoints of the edge with $u$ and $v$. Strictly speaking this requires a choice of orientation on each edge of $G$, but this will be irrelevant for our applications. We often write $G[H]$ when the terminals are clear from context. 

To study the effect of such a gadget implementation on the reliability polynomial, the \emph{split reliability polynomial} of a two-terminal graph $(H,u,v)$ was introduced in~\cite{brown2017roots}. This polynomial is denoted by $\sRel{H}{u,v}{q}$ and represents the probability that the operational edges induce exactly two components with $u$ and $v$ in different components:
\[
	\sRel{H}{u,v}{q} = \sum_{S} (1-q)^{|S|} q^{|E(H)| - |S|},
\]
where the sum ranges over all $S \subseteq E(H)$ such that $(V(H), S)$ has exactly two components with $u$ and $v$ in different components. We simply write $\sRel{H}{}{q}$ when the terminals are clear from context. The effect of an edge substitution is to replace the failure probability $q$ by an effective failure probability determined by the gadget~\cite[Eq.~(2.3)]{bencs2025zeros}:
\[
	\Rel{G[H]}{q} = \bigl(\Rel{H}{q}+\sRel{H}{}{q}\bigr)^{|E(G)|} \cdot \operatorname{Rel}\!\left(G;\, \frac{\sRel{H}{}{q}}{\Rel{H}{q}+\sRel{H}{}{q}}\right).
\]

\subsection{Proof strategy}\label{sec:strategy}
Let $K_2^{\parallel n}$ denote the multigraph consisting of two vertices and $n$ parallel edges. Brown and Colbourn~\cite{brown1992roots} used the family of these gadgets implemented in cycles $C_m[K_2^{\parallel n}]$ to prove the multigraph version of \autoref{thm:main}. In the present paper, the role of $K_2^{\parallel n}$ is played by the complete graph $K_{n+1}$, making the resulting implemented graphs simple; see \autoref{fig:constructions}. 

The reason that $K_{n+1}$ can play the same role as $K_2^{\parallel n}$ is that both families share the same asymptotic behavior for large $n$ when $q$ lies inside the disk. Namely, for both families, $\Rel{H_n}{q}$ converges to $1$ and $\sRel{H_n}{}{q} \sim c \cdot q^n$ for a nonzero constant $c$; see \autoref{lem:limKn2}. We show in \autoref{lem:rootslemma} that this is sufficient to prove density of zeros.

\begin{figure}[ht]
\centering
\begin{tabular}{c@{\qquad\qquad}c}
\begin{tikzpicture}[
    main/.style={circle, fill=black, inner sep=1.2pt},
    baseline={(0,0)}
]

\def\n{7}
\def\k{5}
\def\radius{2}

\foreach \i in {1,...,\n} {
    \pgfmathsetmacro{\angle}{90 + (\i-1)*360/\n}
    \coordinate (v\i) at (\angle:\radius);
}

\foreach \i in {1,...,\n} {
    \pgfmathtruncatemacro{\j}{mod(\i,\n)+1}
    \foreach \e in {1,...,\k} {
        \pgfmathsetmacro{\bendangle}{(\e - (\k+1)/2) * 12}
        \draw[black, semithick] (v\i) to[bend left=\bendangle] (v\j);
    }
}

\foreach \i in {1,...,\n} {
    \node[main] at (v\i) {};
}

\end{tikzpicture}%
 & \begin{tikzpicture}[
    main/.style={circle, fill=black, inner sep=1.2pt},
    internal/.style={circle, fill=black, inner sep=1.2pt},
    baseline={(0,0)}
]

\def\n{7}
\def\radius{1.6}
\def\pscale{0.5}

\foreach \i in {1,...,\n} {
    \pgfmathsetmacro{\angle}{90 + (\i-1)*360/\n}
    \coordinate (v\i) at (\angle:\radius);
}

\foreach \i in {1,...,\n} {
    \pgfmathtruncatemacro{\j}{mod(\i,\n)+1}
    \pgfmathsetmacro{\ai}{90 + (\i-1)*360/\n}
    \pgfmathsetmacro{\aj}{90 + (\j-1)*360/\n}
    \pgfmathsetmacro{\xi}{\radius*cos(\ai)}
    \pgfmathsetmacro{\yi}{\radius*sin(\ai)}
    \pgfmathsetmacro{\xj}{\radius*cos(\aj)}
    \pgfmathsetmacro{\yj}{\radius*sin(\aj)}
    \pgfmathsetmacro{\mx}{(\xi+\xj)/2}
    \pgfmathsetmacro{\my}{(\yi+\yj)/2}
    \pgfmathsetmacro{\dx}{\xj-\xi}
    \pgfmathsetmacro{\dy}{\yj-\yi}
    \pgfmathsetmacro{\s}{sqrt(\dx*\dx+\dy*\dy)}
    \pgfmathsetmacro{\sr}{\s/2*\pscale}
    \pgfmathsetmacro{\tx}{\dx/\s}
    \pgfmathsetmacro{\ty}{\dy/\s}
    \pgfmathsetmacro{\mlen}{sqrt(\mx*\mx+\my*\my)}
    \pgfmathsetmacro{\nx}{\mx/\mlen}
    \pgfmathsetmacro{\ny}{\my/\mlen}
    \pgfmathsetmacro{\ra}{\s/2}
    \foreach \k in {1,...,4} {
        \pgfmathsetmacro{\theta}{\k*180/5}
        \pgfmathsetmacro{\wx}{\mx + \ra*(cos(\theta)*\tx + sin(\theta)*\nx)}
        \pgfmathsetmacro{\wy}{\my + \ra*(cos(\theta)*\ty + sin(\theta)*\ny)}
        \coordinate (w\i-\k) at (\wx, \wy);
    }
    \draw[black, semithick] (v\i) -- (v\j);
    \draw[black, thin] (v\i) -- (w\i-1);
    \draw[black, thin] (v\i) -- (w\i-2);
    \draw[black, thin] (v\i) -- (w\i-3);
    \draw[black, thin] (v\i) -- (w\i-4);
    \draw[black, thin] (v\j) -- (w\i-1);
    \draw[black, thin] (v\j) -- (w\i-2);
    \draw[black, thin] (v\j) -- (w\i-3);
    \draw[black, thin] (v\j) -- (w\i-4);
    \draw[black, thin] (w\i-1) -- (w\i-2);
    \draw[black, thin] (w\i-1) -- (w\i-3);
    \draw[black, thin] (w\i-1) -- (w\i-4);
    \draw[black, thin] (w\i-2) -- (w\i-3);
    \draw[black, thin] (w\i-2) -- (w\i-4);
    \draw[black, thin] (w\i-3) -- (w\i-4);
    \node[internal] at (w\i-1) {};
    \node[internal] at (w\i-2) {};
    \node[internal] at (w\i-3) {};
    \node[internal] at (w\i-4) {};
}

\foreach \i in {1,...,\n} {
    \node[main] at (v\i) {};
}

\end{tikzpicture}%
 \\[0.5cm]
$C_7[K_2^{\parallel 5}]$ & $C_7[K_6]$
\end{tabular}
\caption{Left: an example of the family of graphs $C_m[K_2^{\parallel n}]$ used in~\cite{brown1992roots} to prove the density of reliability roots for multigraphs. Right: an example of the family $C_m[K_{n+1}]$ used in this paper to prove the analogous result for simple graphs.}
\label{fig:constructions}
\end{figure}
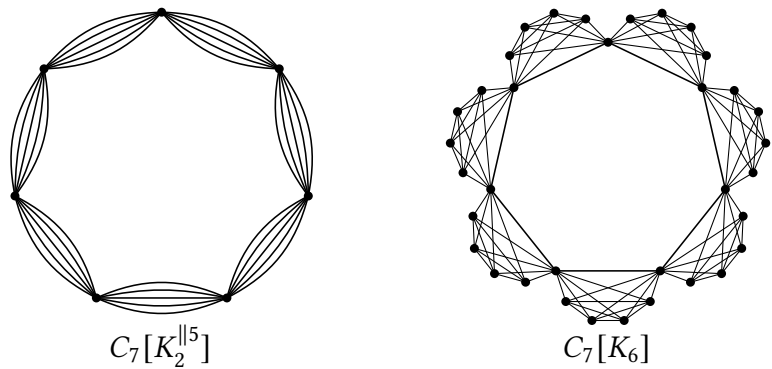

\subsection{Formalization}
The density of real reliability roots of simple graphs in $[-1,0]$ (\autoref{thm:main}, item 2) has been formally verified in the Lean~4 proof assistant. The formalization is available at \url{https://github.com/Pjotr5/ReliabilityRoots}. The main obstacle to formalizing item~1 is that Rouch\'{e}'s theorem has not yet been verified in Lean; the author is currently working on adding this.

\section{The (split) reliability polynomial of a complete graph} 
In this section we determine the asymptotic behavior of both the reliability polynomial and the split reliability polynomial of the complete graph $K_n$ as $n \to \infty$ for $q$ strictly inside the unit disk. As a consequence, we show that the real reliability roots of complete graphs accumulate at $-1$.
\begin{lemma}
	\label{lem:limKn2}
	Let $q \in \mathbb{C}$ with $|q| < 1$. Then
	\[
		\lim_{n \to \infty} \Rel{K_n}{q}  = 1 
		\quad\text{ and }\quad \lim_{n \to \infty} \sRel{K_n}{}{q} / q^{n-1} = 2.
	\]
	The convergence is uniform on compact subsets of the open unit disk $\mathbb{D}$.
\end{lemma}
In the statement of this lemma we omit a specific choice of terminals for the split reliability polynomial, which of course can be taken as any distinct pair. Note that for $q \in [0,1]$, the value $\Rel{K_n}{q}$ equals the probability that the Erd\H{o}s--R\'{e}nyi--Gilbert random graph $G(n,1-q)$ is connected. This model was introduced by Gilbert~\cite{gilbert1959random}, who proved that $\Rel{K_n}{q} = 1 - nq^{n-1} + O(n^2 q^{3n/2})$.

\subsection{Recursive identities}\label{sec:recursions}
Conditioning on the vertex set $S$ of the component of a fixed vertex in the random subgraph of $K_n$ (where each edge fails independently with probability $q$), and noting that $K_n$ is disconnected precisely when $|S| < n$, we obtain the following recursion, which appeared in \cite[Eq.~(4)]{gilbert1959random}:
\begin{equation}\label{eq:Rn-rec2}
	\Rel{K_n}{q} = 1 - \sum_{s=1}^{n-1} \binom{n-1}{s-1} \Rel{K_s}{q}\, q^{s(n-s)}.
\end{equation}
Indeed, there are $\binom{n-1}{s-1}$ choices for $S$ with $|S| = s$, the induced subgraph $K_n[S] \cong K_s$ must be connected, and all $s(n-s)$ edges between $S$ and $V \setminus S$ must fail.

Similarly, conditioning on the vertex set $S$ of the component of one terminal and noting that $K_n$ is split precisely when both $S$ and its complement induce connected subgraphs, we obtain:
\begin{equation}\label{eq:Sn-rec}
	\sRel{K_n}{}{q} = \sum_{s=1}^{n-1} \binom{n-2}{s-1} \Rel{K_s}{q}\, \Rel{K_{n-s}}{q}\, q^{s(n-s)}.
\end{equation}
Here there are $\binom{n-2}{s-1}$ choices for $S$ with $|S| = s$ (the other terminal is excluded), $K_s$ and $K_{n-s}$ must each be connected, and all $s(n-s)$ edges between $S$ and $V \setminus S$ must fail. A corresponding formula for $K_n^-$ (the complete graph with the edge between the two terminals removed) appeared in \cite[Eq.~(8)]{brown2017roots}.

\subsection{Proof of \autoref{lem:limKn2}}
The proof of \autoref{lem:limKn2} reduces to showing that the sum $a_n$ defined in~\eqref{eq:anbn} below tends to zero. The sequence $a_n$ also appears in~\cite{gilbert1959random}, where its asymptotic behavior is determined more precisely.
\begin{lemma}\label{lem:anbn}
Let $\rho \in [0,1)$. The sequences
\begin{equation}
\label{eq:anbn}
a_n=\sum_{s=1}^{n-1}\binom{n-1}{s-1}\rho^{\,s(n-s)}
\quad\text{and}\quad
b_n=\sum_{s=2}^{n-2}\binom{n-2}{s-1}\rho^{\,s(n-s)-(n-1)}
\end{equation}
both converge to zero as $n \to \infty$.
\end{lemma}
\begin{proof}
	We rewrite $b_{n+2}$ as
	\[
		b_{n+2} = \sum_{s=2}^{n}\binom{n}{s-1}\rho^{\,s(n+2-s)-(n+1)} = \sum_{s=1}^{n-1}\binom{n}{s}\rho^{\,s(n-s)}.
	\]
	The summand is invariant under $s \mapsto n-s$, so
	\[
	\sum_{s=1}^{n-1}\binom{n}{s}\rho^{\,s(n-s)} \leq 2\sum_{s=1}^{\lfloor n/2\rfloor}\binom{n}{s}(\rho^{n/2})^s \leq 2\sum_{s=1}^{n}\binom{n}{s}(\rho^{n/2})^s
	= 2\bigl[(1 + \rho^{n/2})^n - 1\bigr] \leq 2\bigl[\exp\bigl(n \cdot \rho^{n/2}\bigr) - 1\bigr].
	\]
	Since $n\rho^{n/2}\to 0$, the right-hand side converges to zero, so $b_n \to 0$. For $a_n$, note that since $(s+1)(n-s) \ge s(n-s)$,
	\[
	a_{n+1} = \sum_{s=1}^{n}\binom{n}{s-1}\rho^{\,s(n+1-s)} = \sum_{s=0}^{n-1}\binom{n}{s}\rho^{\,(s+1)(n-s)} \leq \rho^n + \sum_{s=1}^{n-1}\binom{n}{s}\rho^{\,s(n-s)} = \rho^n + b_{n+2} \to 0. \qedhere
	\]
\end{proof}
\begin{proof}[Proof of \autoref{lem:limKn2}]
Let $\rho \in (0,1)$ and let $a_n$, $b_n$ be as in \eqref{eq:anbn}. We establish uniform convergence on the closed disk of radius $\rho$, which we denote by $B_\rho$. Let $N$ be sufficiently large such that $a_n \leq \frac{1}{2}$ for $n \geq N$, and let
	\[
		M = \max_{q \in B_\rho}\bigl(2,\, |\Rel{K_1}{q}|,\, \dots,\, |\Rel{K_N}{q}|\bigr).
	\]
	Inductively we have $|\Rel{K_n}{q}| \leq M$ for all $n$ and $q\in B_\rho$, since for $n > N$ \autoref{eq:Rn-rec2} gives
	\[
		|\Rel{K_n}{q}| \leq 1 + a_n \cdot M \leq 1 + \frac{1}{2} M \leq M.
	\]
	Therefore $|1 - \Rel{K_n}{q}| \leq a_n \cdot M \to 0$ as $n \to \infty$. This proves the limit statement for $\Rel{K_n}{q}$. 
	
	We use \autoref{eq:Sn-rec} to see that
\[
	\sRel{K_n}{}{q} / q^{n-1} = 2\Rel{K_{n-1}}{q} +  \sum_{s=2}^{n-2} \binom{n-2}{s-1} \Rel{K_s}{q}\, \Rel{K_{n-s}}{q}\, q^{s(n-s)-(n-1)}.
\]
Because $\Rel{K_{n-1}}{q} \to 1$, the first term converges to $2$ uniformly. We see that the remaining sum is bounded above by
\[
	\left| \sum_{s=2}^{n-2} \binom{n-2}{s-1} \Rel{K_s}{q}\, \Rel{K_{n-s}}{q}\, q^{s(n-s)-(n-1)}\right| \leq M^2 \cdot b_n,
\]
which converges to zero as $n \to \infty$.
\end{proof}
It was remarked in \cite{brown2026realreliabilityrootsgraphs} that the negative real roots of a subsequence of the complete graphs computationally appear to accumulate at $-1$. We can now prove this.
\begin{corollary}
	Let $q \in (-1,0)$. For any sufficiently large $n \equiv 0$ or $3 \pmod{4}$, the polynomial $\Rel{K_n}{\cdot}$ has a zero in the interval $(-1,q)$.
\end{corollary}
\begin{proof}
	It is proven in \cite[Thm.~3.1]{brown2020rational} that for any simple graph with $n$ vertices and $m$ edges the value of $\Rel{G}{-1}$ is nonzero and its sign is $(-1)^{m-n+1}$. It follows that the sign of $\Rel{K_n}{-1}$ is $(-1)^{\binom{n}{2}-n+1}$, which is negative for $n \equiv 0$ or $3 \pmod{4}$. By \autoref{lem:limKn2}, $\Rel{K_n}{q} > 0$ for all sufficiently large $n$, so the intermediate value theorem gives a zero in $(-1,q)$.
\end{proof}

\section{Proof of the main theorem}
Recall from \autoref{sec:strategy} that $C_m[K_{n+1}]$ denotes the graph obtained by replacing each edge of the cycle $C_m$ with a copy of $K_{n+1}$. This graph is simple and connected for $m \geq 3$ and $n \geq 1$. Since a spanning subgraph of $C_m[K_{n+1}]$ is connected if and only if every copy of $K_{n+1}$ connects its terminals, or exactly one copy splits while all others connect, we have
\begin{equation}\label{eq:rel-cycle}
	\Rel{C_m[K_{n+1}]}{q} = \Rel{K_{n+1}}{q}^{m-1}\bigl(\Rel{K_{n+1}}{q} + m \cdot \sRel{K_{n+1}}{}{q}\bigr).
\end{equation}
In particular, every zero of $\Rel{K_{n+1}}{\cdot} + m \cdot \sRel{K_{n+1}}{}{\cdot}$ is a reliability root of the simple graph $C_m[K_{n+1}]$. Density of zeros will now follow from the following lemma.

\begin{lemma}
	\label{lem:rootslemma}
	Let $\{f_n\}_{n \geq 1}, \{g_n\}_{n \geq 1}$ be sequences of holomorphic functions on the unit disk $\mathbb{D}$, and let $c \in \mathbb{C} \setminus \{0\}$. Suppose
	the following limits hold for $z \in \mathbb{D}$:
	\[
		\lim_{n \to \infty} f_n(z)  = 1
		\quad\text{ and }\quad \lim_{n \to \infty} g_n(z) / z^n = c,
	\]
	and the convergence is uniform on compact subsets of $\mathbb{D}$. Then the zero set
	\[
		\mathcal{Z} := \{z \in \mathbb{D}: f_n(z) + m \cdot g_n(z) = 0 \text{ for some integers $n \geq 1$ and $m \geq 3$} \}
	\]
	is dense in $\mathbb{D}$. Moreover, if $f_n$ and $g_n$ are real-valued on $(-1,0)$ then $\mathcal{Z} \cap (-1,0)$ is dense in $[-1,0]$.	
\end{lemma}

\begin{proof}
	Pick an arbitrary $z_0 \in \mathbb{D} \setminus \{0\}$. We show that elements of $\mathcal{Z}$ accumulate on $z_0$.
	Write $\rho = |z_0|$ and define the integers $m_n = \lfloor 1/(|c| \cdot \rho^n) \rfloor$. Note that $m_n \to \infty$ and in particular $m_n \geq 3$ for all sufficiently large $n$.
	Define the sequence of polynomials $p_n(z) = 1 + m_n c z^n$. The zeros of $p_n$ are $n$ equally spaced points on a circle of radius $\rho_n := (m_n |c|)^{-1/n}$.
	Since $\rho_n \to \rho$, there is a sequence $z_n$ of zeros of $p_n$ with $z_n \to z_0$. It now suffices to find $z_n^*$ with $f_n(z_n^*) + m_n g_n(z_n^*) = 0$ and $|z_n - z_n^*| \to 0$.
	
	Define the closed disks $B_n = \{z_n (1 + u) : |u| \leq 1/n \}$. The relation $m_n c z_n^n = -1$ gives $p_n(z_n (1 + u)) = 1 - (1+u)^n$. For $|u| = 1/n$,
	\[
		|(1+u)^n - 1| = \bigl| nu + \textstyle\sum_{k=2}^n \binom{n}{k}u^k\bigr| \geq 1 - \textstyle\sum_{k=2}^n \binom{n}{k}\bigl(\tfrac{1}{n}\bigr)^k = 1 - \bigl((1+\tfrac{1}{n})^n-2\bigr) \geq 3 - e.
	\]
	It follows that for $z \in \partial B_n$ we have the uniform lower bound $|p_n(z)| \geq 3-e$. 
	
	For $n$ sufficiently large the sets $B_n$ lie inside some strictly smaller closed disk $B \subset \mathbb{D}$. The uniform convergence on $B$ gives bounds
	$|f_n(z) - 1| < \alpha_n$ and $|g_n(z) / z^n  - c| < \beta_n$ with $\alpha_n, \beta_n \to 0$. For $z \in \partial B_n$, writing $z = z_n(1+u)$ with $|u| = 1/n$ gives $m_n |z|^n = |1+u|^n/|c| \leq e/|c|$, so
	\[
		|f_n(z) + m_n g_n(z)  - p_n(z)| \leq |f_n(z) - 1| + m_n |z|^n \cdot |g_n(z) / z^n  - c| \leq \alpha_n + \tfrac{e}{|c|} \beta_n.
	\]
	For sufficiently large $n$ this is less than $3 - e$, so by Rouch\'{e}'s theorem $f_n + m_n g_n$ has a zero inside $B_n$. This completes the proof of the first statement.

	For the moreover statement, note that $c$ is necessarily real since $g_n(x)/x^n \to c$ for real $x$. Now let $z_0 \in (-1,0)$ and define $\rho, m_n, p_n$ as above. For infinitely many $n$ (odd $n$ if $c > 0$, even $n$ if $c < 0$), $p_n$ has a real zero $z_n = -(m_n|c|)^{-1/n} \to z_0$ in $(-1,0)$. At the real endpoints $z_n(1 \pm 1/n)$ of $B_n \cap \mathbb{R}$ the values $p_n(z_n(1 \pm 1/n)) = 1 - (1 \pm 1/n)^n$ converge to $1 - e$ and $1 - 1/e$ respectively, which have opposite signs. Since the error $|f_n + m_n g_n - p_n| \to 0$ uniformly, the function $f_n + m_n g_n$ also changes sign on this interval for large $n$, and the intermediate value theorem gives a real zero.
\end{proof}

We can now prove the main result, which we restate for convenience.
\maintheorem*
\begin{proof}
	By~\eqref{eq:rel-cycle}, the zeros of $\Rel{K_{n+1}}{\cdot} + m \cdot \sRel{K_{n+1}}{}{\cdot}$ are reliability roots of the simple graph $C_m[K_{n+1}]$. \autoref{lem:rootslemma} applied with $f_n(q) = \Rel{K_{n+1}}{q}$ and $g_n(q) = \sRel{K_{n+1}}{}{q}$ shows that these roots are dense in $\mathbb{D}$. The moreover statement of \autoref{lem:rootslemma} gives density of real reliability roots in $[-1,0]$. The closure $\overline{\mathcal{Z}_{\mathrm{simple}} \cap \mathbb{R}} = [-1,0] \cup \{1\}$ then follows from \cite[Cor.~3.2]{brown1992roots}, which states that the real zeros of the reliability polynomial of any connected multigraph are contained in $[-1,0) \cup \{1\}$.
\end{proof}

\begin{remark}
	The original proof of~\cite[Prop.~5.1]{brown1992roots} uses the family $C_m[K_2^{\parallel n}]$ to prove density of roots of multigraphs; see \autoref{fig:constructions}. This proof can also be recovered from \autoref{lem:rootslemma}, since $\Rel{K_2^{\parallel n}}{q} = 1 - q^n$ and $\sRel{K_2^{\parallel n}}{}{q} = q^n$ clearly satisfy the asymptotic conditions.
\end{remark}

\section{Open problems}
We conclude with two related open problems.
\subsection{Zeros of the complete graph}
As this paper shows, constructions involving complete graphs are useful to find examples of graphs with roots anywhere inside the unit disk. However, computations suggest that the complete graphs themselves have no reliability roots outside the unit disk; see \autoref{fig:roots-K25}. This leads to the following question.
\begin{question}\label{q:Kn-roots}
	Suppose $\Rel{K_n}{q} = 0$ for some $n$ and $q \neq 1$. Does it follow that $|q| < 1$?
\end{question}
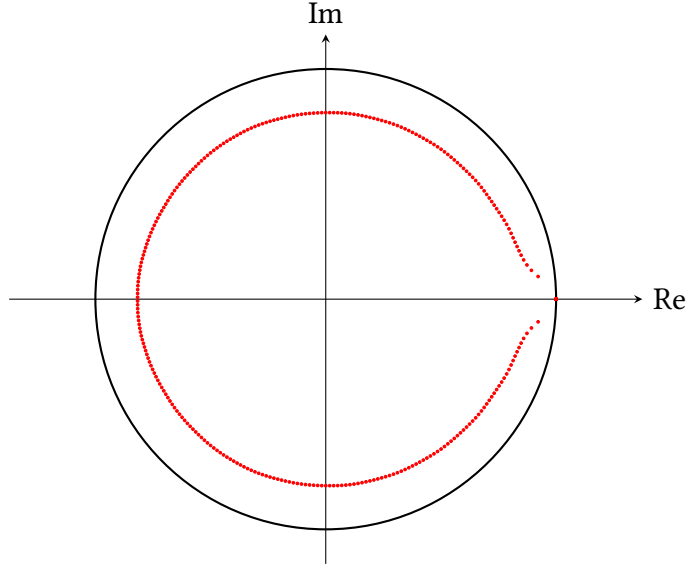
\begin{figure}[ht]
\centering
\begin{tikzpicture}
  \begin{axis}[
    axis equal,
    width=10cm,
    xmin=-1.15, xmax=1.15,
    ymin=-1.15, ymax=1.15,
    axis lines=middle,
    xlabel={$\operatorname{Re}$},
    ylabel={$\operatorname{Im}$},
    every axis x label/.style={at={(ticklabel* cs:1)}, anchor=west},
    every axis y label/.style={at={(ticklabel* cs:1)}, anchor=south},
    xtick=\empty,
    ytick=\empty,
  ]
    \draw[black, thick] (axis cs:0,0) circle[radius=1];
    \addplot[only marks, mark=*, mark size=0.5pt, red] table[col sep=comma, header=false] {roots_K25.csv};
  \end{axis}
\end{tikzpicture}
\caption{The zeros of $\Rel{K_{25}}{q}$ in the complex plane, together with the unit circle.}
\label{fig:roots-K25}
\end{figure}
Note that \autoref{lem:limKn2} implies that for any $\rho < 1$, the closed disk of radius $\rho$ is zero-free for all sufficiently large $n$. A positive answer to \autoref{q:Kn-roots} would therefore mean that the reliability roots of $K_n$ approach the unit circle as $n \to \infty$. Recall that the \emph{zero-counting measure} of a polynomial $p$ of degree $d$ with zeros $z_1, \ldots, z_d$ (counted with multiplicity) is $\mu_p = \frac{1}{d} \sum_{i=1}^{d} \delta_{z_i}$. It is natural to ask whether the zero-counting measures of $\Rel{K_n}{\cdot}$ converge weakly to the uniform measure on the unit circle.

\subsection{Closure of simple versus multigraph roots}
The zero sets $\mathcal{Z}_{\mathrm{simple}}$ and $\mathcal{Z}_{\mathrm{multi}}$ are not equal: for instance, $q = -1$ is a root of $\Rel{K_2^{\parallel n}}{q} = 1 - q^n$ for every even $n$, but $\Rel{G}{-1} \neq 0$ for every simple graph $G$~\cite{brown2020rational}. However, \autoref{thm:main} together with~\cite[Cor.~3.2]{brown1992roots} shows that the closures of their real parts coincide: $\overline{\mathcal{Z}_{\mathrm{simple}} \cap \mathbb{R}} = \overline{\mathcal{Z}_{\mathrm{multi}} \cap \mathbb{R}} = [-1,0] \cup \{1\}$.
\begin{question}
	Does the closure of the reliability roots of simple graphs equal the closure of the reliability roots of multigraphs, i.e.\ is $\overline{\mathcal{Z}_{\mathrm{simple}}} = \overline{\mathcal{Z}_{\mathrm{multi}}}$?
\end{question}
We note that the answer is yes if $\mathcal{Z}_{\mathrm{multi}}$ is unbounded. It is shown in~\cite[Prop.~1.3]{bencs2025zeros} that either $\mathcal{Z}_{\mathrm{multi}}$ is bounded or $\overline{\mathcal{Z}_{\mathrm{multi}}} = \mathbb{C}$; it remains open which of the two holds. Subdividing every edge of a multigraph $G$ yields a simple graph $G[P_3]$ whose reliability satisfies (cf.\ \autoref{sec:splitrel} or~\cite[Sec.~5]{brown1992roots})
\[
	\Rel{G[P_3]}{q} = (1-q^2)^{|E(G)|} \cdot \operatorname{Rel}\!\left(G;\, \frac{2q}{1+q}\right).
\]
Thus if $q_0$ is a reliability root of a multigraph, then $q_0/(2-q_0)$ is a reliability root of a simple graph. Since $q_0 \mapsto q_0/(2-q_0)$ is a M\"{o}bius transformation, it maps any dense subset of $\mathbb{C}$ to a dense subset, giving $\overline{\mathcal{Z}_{\mathrm{simple}}} = \mathbb{C} = \overline{\mathcal{Z}_{\mathrm{multi}}}$.

\section*{Acknowledgements}
The author thanks Ferenc Bencs for useful discussions.

\bibliographystyle{alpha}
\bibliography{references}

\end{document}